\documentclass{amsart}
\usepackage{amsmath,amssymb,amsfonts}
\usepackage{amsxtra}
\usepackage{mathtools}
\usepackage{stmaryrd}
\usepackage[all]{xy}
\usepackage[normalem]{ulem}
\usepackage{color}
\usepackage[colorlinks=true]{hyperref}

\DeclareMathOperator{\rank}{rank}
\DeclareMathOperator{\discr}{discr}
\DeclareMathOperator{\trace}{tr}
\DeclareMathOperator{\Aut}{Aut}

\DeclareMathOperator{\Orth}{O}

\DeclareMathOperator{\id}{id}
\DeclareMathOperator{\Hom}{Hom}
\DeclareMathOperator{\Pic}{Pic}

\DeclareMathOperator{\Tors}{T}

\DeclareMathOperator{\diff}{d}
\DeclareMathOperator{\Sym}{S}
\DeclareMathOperator{\Trans}{T}

\newcommand{\ii}{{\rm i}}
\newcommand{\dual}[1]{{#1}\spcheck}
\newcommand{\abs}[1]{|{#1}|}
\newcommand{\symm}[2]{{#2}^{(#1)}}
\newcommand{\hilb}[2]{{#2}^{[#1]}}
\newcommand{\ps}[2]{#2^{\llbracket#1\rrbracket}}

\newcommand{\HH}{{\rm H}}

\newcommand{\IC}{\mathbb{C}}
\newcommand{\IP}{\mathbb{P}}
\newcommand{\IR}{\mathbb{R}}
\newcommand{\IQ}{\mathbb{Q}}
\newcommand{\IZ}{\mathbb{Z}}

\newcommand{\cL}{\mathcal{L}}
\newcommand{\cO}{\mathcal{O}}

\newcommand{\nn}{\mathfrak{n}}

\newcommand{\genKm}{\ps n A}
\newcommand{\hilbA}{\hilb n A}

\newtheorem{theorem}{Theorem}[section]

\newtheorem{proposition}[theorem]{Proposition}
\newtheorem{corollary}[theorem]{Corollary}

\theoremstyle{definition}

\newtheorem{example}[theorem]{Example}

\theoremstyle{remark}
\newtheorem{remark}[theorem]{Remark}

\numberwithin{equation}{section}

\begin{document}

\title[Hyper-K\"ahler fourfold symmetries]
{Some algebraic and geometric properties of hyper-K\"ahler fourfold symmetries}

\author{Samuel Boissi\`ere}

\address{Samuel Boissi\`ere,
Laboratoire de Math\'ematiques et Applications,
UMR 7348  CNRS,
B\^atiment H3,
Boulevard Marie et Pierre Curie,
Site du Futuroscope,
TSA 61125,
86073 Poitiers Cedex 9,
France}
			
\email{samuel.boissiere@math.univ-poitiers.fr}

\urladdr{http://www-math.sp2mi.univ-poitiers.fr/$\sim$sboissie/}

\author{Marc Nieper-Wi{\ss}kirchen}

\address{Marc Nieper-Wi{\ss}kirchen, Lehrstuhl f\"ur Algebra und Zahlentheorie, Universit\"ats\-stra{\ss}e~14, 86159 Augsburg, Germany}
			
\email{marc.nieper-wisskirchen@math.uni-augsburg.de}

\urladdr{http://www.math.uni-augsburg.de/alg/}

\author{K\'evin Tari}

\address{K\'evin Tari, Lyc\'ee Henri Moissan, 20 Cours de Verdun, 77100 Meaux}
			
\email{kevin.tari@free.fr}

\date{\today}



\begin{abstract}
We  complete the classification of order $5$ nonsymplectic automorphisms on hyper-K\"ahler fourfolds deformation equivalent to the Hilbert square of a K3 surface. We then compute the topological Lefschetz number of natural automorphisms of generalized Kummer fourfolds and we describe the geometry of their fix loci.
\end{abstract}

\maketitle

\section{Introduction}

Hyper-K\"ahler manifolds, or equivalently irreducible holomorphic symplectic (IHS) manifolds, are traditionally labelled by their deformation type, which encodes
most of the significant features of the manifold. Currently, four deformation types have been exhibited, named after a representative of each class: the Hilbert scheme of $n$~points on a K3 surface, for any $n\geq 1$; the $n$-th generalized Kummer variety of an abelian surface, for any $n\geq 2$; the O'Grady variety OG6; the O'Grady variety OG10. The present paper focuses on hyper-K\"ahler fourfolds of Hilbert and of Kummer type, and the topic of interest concerns the classification of the biholomorphic automorphisms of these manifolds. 

As for K3 surfaces, the approach for the classification of automorphisms of IHS manifolds is twofold: first a lattice-theoretical classification of the invariant sublattice and of its orthogonal complement inside the second cohomology space with integer coefficients, endowed with the Beauville--Bogomolov--Fujiki quadratic form; then a description of the fix locus of the automorphism which, whenever it is nonempty, is smooth and has usually several connected components of different dimensions.

In the last years, many authors have contributed to the classification of prime order automorphisms on the two families of IHS manifolds in issue here.
We refer to~\cite[\S 1]{BCS_class},\cite[\S 1]{MTW_Aut} and references therein for a more detailed introduction of the contribution of each author. So far the classification of order $5$ nonsymplectic automorphisms on the family of the Hilbert square of a K3 surface had still remained incomplete due to the lack of an algebraic ingredient (Theorem~\ref{th:square}).
Using this result, in Theorem~\ref{th:class} we complete the classification of order $5$ nonsymplectic automorphisms on the deformation class of the Hilbert square of a K3 surface.

In the second part of this note, we describe the fix locus of  \emph{natural} automorphisms on generalized Kummer fourfolds, which are the automorphisms coming from the underlying abelian surface. We first give in Proposition~\ref{prop:formule} a generating formula for the topological Lefschetz number of these automorphisms. A similar formula already appeared in~\cite{BNWS_Enriques} but the formula given there is not correct when the automorphism contains a nontrivial translation. We correct  the formula here and we give a complete proof, which requires substantial work to take care of  the missing factor.
As an application, we discuss in Section~\ref{ss:appl} the fix loci of natural automorphisms on generalized Kummer fourfolds, whose action on cohomology has prime order.

\subsection*{Acknowledgements}
The authors warmly thank Simon Brandhorst, Chiara Camere, Paolo Menegatti and Alessandra Sarti for helpful comments, enlightening discussions and precious help during the preparation of this work. We express all our gratitude to the anonymous referee for their very careful reading and valuable observations.

\section{Prime order isometries of integral lattices}\label{s:prime}

A \emph{lattice}~$L$ is a free $\IZ$-module equiped with a nondegenerate symmetric bilinear form $\langle\cdot,\cdot\rangle_L$ with integer values. Its \emph{dual lattice} is $L^\vee\coloneqq\Hom_\IZ(L,\IZ)$. Clearly $L$~is a sublattice of $L^\vee$ of equal rank, so the \emph{discriminant group}~$D_L\coloneqq L^\vee/L$ is a finite abelian group, whose order~$\discr(L)$ is called the \emph{discriminant} of~$L$. The lattice~$L$ is called \emph{unimodular} if $\discr(L)=1$. A sublattice $M\subset L$ is called \emph{primitive} if the quotient~$L/M$ is a free $\IZ$-module; it is called $p$-\emph{elementary} for some prime number~$p$ if $D_M\cong(\IZ/p\IZ)^{\oplus a}$ for some positive integer~$a$.

In the sequel, we will use the lattices $E_8(-1)$ and $A_4(-1)$, which are the opposites of the usual positive definite root lattices, the hyperbolic plane $U$, the  rank two lattice $H_5$ with
Gram matrix $\begin{pmatrix} 2 & 1\\1 & -2\end{pmatrix}$, the lattice $A_4^\ast(-5)$ which is the dual lattice of $A_4$ with intersection form multiplied by~$-5$, and the rank one lattice~$\langle -2\rangle$  generated by an element of square~$-2$. Notation $\frac{\IZ}{p\IZ}(q)$, with $q\in \IQ/2\IZ$, means a finite quadratic form generated by an element of square~$q$, we refer to Nikulin~\cite{Nikulin} for deeper reading.

Let $\varphi\in\Orth(L)$ be an isometry of~$L$, of prime order $p$. The invariant subspace: 
$$
T_\varphi\coloneqq \ker(\varphi-\id_L),
$$ 
is a primitive sublattice of~$L$, since the restriction of the bilinear form to $T_\varphi$ is nondegenerate. Its orthogonal complement: 
$$
S_\varphi\coloneqq T^\perp,
$$ 
is also a primitive sublattice of~$L$. Denote by~$\xi_p$ a primitive $p$-th root of unity, by $K\coloneqq\IQ(\xi_p)$ the $p$-th cyclotomic field and by $\cO_K=\IZ[\xi_p]$ its ring of algebraic integers. 
It is easy to check that $S_\varphi=\ker\left(\Phi_p(\varphi)\right)$, where $\Phi_p$ is the $p$-th cyclotomic polynomial, so $S_\varphi\otimes_\IZ\IQ$ is endowed with the structure of a finite-dimensional $K$-vector space. As a consequence, there exists a nonnegative integer $m_\varphi$ such that:
$$
\rank_\IZ S_\varphi=\dim_\IQ S_\varphi\otimes_\IZ\IQ=(p-1)m_\varphi.
$$

\begin{proposition}  \label{prop:rel_a_m}
Let $L$ be an integral lattice and $\varphi\in\Orth(L)$ an isometry of
prime order~$p$. There exists a nonnegative integer~$a_\varphi$ such that $\frac{L}{T_\varphi\oplus S_\varphi}\cong \left(\frac{\IZ}{p\IZ}\right)^{\oplus a_\varphi}$ and we have $a_\varphi\leq m_\varphi$.
\end{proposition}

It is not difficult to see that $\frac{L}{T_\varphi\oplus S_\varphi}$ is a $p$-torsion module (see~\cite[Lemma~3.1]{BCMS}, \cite[Lemme~2.9]{Tari}). The property $a_\varphi\leq m_\varphi$ was first proved in~\cite[Theorem~2.1(c)]{AST} in the context of isometries of K3 lattices and then in~\cite[Corollary~3.7]{BCS_class} in the context of isometries of hyper-K\"ahler manifolds of $\text{K3}^{[2]}$-type, but only for {${p\notin\{5,23\}}$}. The above statement is a very useful generalization of these properties, whose proof first appeared in the third author's PhD thesis~\cite[Th\'eor\`eme~2.18]{Tari}, see also~\cite[Lemma~1.8]{MTW_Aut}.

\begin{theorem}\label{th:square} Let $L$ be an integral lattice and $\varphi\in\Orth(L)$ an isometry of prime order~$p\neq 2$. The integer $p^{m_\varphi}\discr(S_\varphi)$ is a square. 
\end{theorem}

This result first appeared in this context in~\cite[Equation~(2)]{BCMS}, in the special case $m_\varphi=1$, based on results of Bayer--Fl\"uckiger~\cite{Bayer}. It was generalized in the third author's PhD thesis~\cite[Th\'eor\`eme~2.23]{Tari} by an induction argument. A more general statement, with a shorter proof, is given by
Bayer--Fl\"uckiger~\cite[Proposition~5.1]{Bayer2}: to apply it to our situation, simply observe that the characteristic polynomial of the restriction of $\varphi$ to $S_\varphi\otimes_\IZ\IQ$ is $\Phi_p^{m_\varphi}$.

\begin{corollary} Under the assumptions of Theorem~\ref{th:square}, if furthermore the lattice~$L$ is unimodular, then $\discr(S_\varphi)=p^{a_\varphi}$ and $a_\varphi\equiv m_\varphi\mod(2)$.
\end{corollary}

\begin{proof} If $L$ is unimodular, then $D_{S_\varphi}\cong D_{T_\varphi}\cong \left(\frac{\IZ}{p\IZ}\right)^{\oplus a_\varphi}$ so $\discr(S_\varphi)=p^{a_\varphi}$ 
(see~\cite[\S I.2]{BHPV}), hence by 
Theorem~\ref{th:square}, the number $a_\varphi+m_\varphi$ is even.
\end{proof}

\section{Order five symmetries of the Hilbert square of a K3 surface}

Let $X$ be a projective irreducible holomorphic symplectic manifold, deformation equivalent to the Hilbert square of a K3 surface. We recall that the second cohomology group with integer coefficients $\HH^2(X,\IZ)$ is a rank $23$ lattice for the Beauville--Bogomolov--Fujiki quadratic form, which is isometric to the lattice:
\[
L\coloneqq U^{\oplus 3}\oplus E_8(-1)^{\oplus 2}\oplus \langle -2\rangle.
\]

Let $f\in\Aut(X)$ be a nonsymplectic biholomorphic automorphism of prime order $p$. It is easy to see that $2\leq p\leq 23$ (see~\cite[\S5.4]{BNWS_Smith}). We consider the isometry $\varphi\coloneqq f^\ast$ induced by $f$ on the lattice $\HH^2(X,\IZ)$. In the sequel we write $m,a,S,T$ instead of $m_\varphi,a_\varphi,S_\varphi,T_\varphi$.
Since the representation:
\[
\Aut(X)\to\Orth(\HH^2(X,\IZ)),\quad f\mapsto (f^\ast)^{-1}
\]
is faithful, nonsymplectic prime order automorphisms are classified by the data of the invariant lattice $T$ of $\varphi$ and its orthogonal complement $S$. For $2\leq p\leq 19$ and $p\neq 5$ the classification is given in~\cite{BCS_class}, for $p=23$ it is given in~\cite{BCMS}. The study of the case $p=5$ was not finished in~\cite[Table~2]{BCS_class} since it was not proven that the list was complete, see~\cite[Remark~6.1]{BCS_class}. The goal of our first main result is to fill this gap: it requires the tools developed in the previous section. 

\begin{theorem}\label{th:class} Let $f$ be an order five nonsymplectic automorphism acting on an irreducible holomorphic symplectic manifold, deformation equivalent to the Hilbert square of a K3 surface. Then its invariant lattice~$T$ and its orthogonal complement~$S$ are one of the following:
$$
\begin{array}{|c|c|c|c|}
\hline
m & a & S & T\\\hline
1 & 1 & U\oplus H_5 & E_8(-1)^{\oplus 2}\oplus H_5\oplus \langle -2\rangle \\\hline
2 & 2 & U\oplus H_5\oplus A_4(-1) & E_8(-1)\oplus H_5\oplus A_4(-1)\oplus\langle -2\rangle\\\hline
3 & 1 & U\oplus E_8(-1)\oplus H_5 & E_8(-1)\oplus H_5\oplus \langle -2\rangle\\\hline
3 & 3 & U\oplus H_5\oplus A_4(-1)^{\oplus 2} & H_5\oplus A_4(-1)^{\oplus 2}\oplus \langle -2\rangle\\\hline
4 & 2 & U\oplus E_8(-1)\oplus H_5\oplus A_4(-1) & H_5\oplus A_4(-1)\oplus\langle -2\rangle\\\hline
4 & 4 & U(5)\oplus E_8(-1)\oplus H_5\oplus A_4(-1) & H_5\oplus A_4^\ast(-5)\oplus \langle -2\rangle\\\hline
5 & 1 & U\oplus E_8(-1)^{\oplus 2}\oplus H_5  & H_5\oplus \langle -2\rangle\\\hline
5 & 3 & U\oplus E_8(-1)\oplus H_5\oplus A_4(-1)^{\oplus 2} & U(5)\oplus\langle -10\rangle\\\hline
\end{array}
$$
\end{theorem}

This result first appeared in the third author's PhD thesis~\cite[Th\'eor\`eme~3.24]{Tari} (note that there is a typo there on case $(4,2)$, that case $(5,3)$ was missing and that in case $(4,4)$ we have  $U(5)\oplus A_4(-1)\cong U\oplus A_4(-5)^\ast$, showing that our result matches what is written there; the lattices reproduced here are those of~\cite[Table~2]{BCS_class}).

\begin{proof}
We start by dressing a list of all possible values of the pair of invariants~$(m,a)$. For this, we use several results from various other sources to narrow down the list of possibilities.
By~\cite[Lemma~5.5]{BNWS_Smith}, the discriminant group of the lattice $S$ is necessarily $D_{S}\cong \left(\IZ/5\IZ\right)^{\oplus a}$, so Theorem~\ref{th:square} gives the relation:
\[
a\equiv m\mod(2).
\] 
By Proposition~\ref{prop:rel_a_m}, we have $a\leq m$ and $4m=\rank_\IZ S\leq 23$, so $1\leq m\leq 5$ and moreover, by~\cite[Proposition~4.13]{BNWS_Smith} we have also $a\leq 23-4m$. This gives a first list of possible pairs $(m,a)$:
$$
(1,1),(2,0),(2,2),(3,1),(3,3),(4,0),(4,2),(4,4),(5,1),(5,3).
$$
Again by~\cite[Lemma~5.5]{BNWS_Smith}, the lattice~$S$ has signature~$(2,4m-2)$, the lattice~$T$ is hyperbolic and $D_{T}\cong \IZ/2\IZ\oplus D_S$. Since: 
\[
\rank_\IZ S=4m\geq 3+m\geq 3+a,
\]
we can apply Nikulin's results on orthogonal decomposition of lattices~\cite[Corollary~1.13.5]{Nikulin}: the lattice $S$ decomposes as $S=U\oplus S'$ for some $5$-elementary lattice~$S'$ of signature~$(1,4m-3)$ and discriminant group~$\left(\IZ/5\IZ\right)^{\oplus a}$. Applying the classification theorem of Rudakov--Shafarevich~\cite[Section~1]{RS}
we get that such a lattice~$S'$ exists for each value of~$(m,a)$ listed aboved, except $(2,0)$ and~$(4,0)$, and that this lattice is unique with these invariants if~$m\geq 2$. For~$m=1$,  uniqueness comes from the classification of binary quadratic forms~\cite[Table 15.2a, p.~362]{CS}.
The genus of the lattice~$T$, which is the orthogonal complement of~$S$ for its unique embedding in~$\HH^2(X,\IZ)$, is characterized by~\cite[Proposition~1.15.1]{Nikulin}. In each case but one, the isometry class of~$T$ is determined 
by~\cite[Corollary~22, p.~395]{CS}.

The only missing case is $(m,a)=(5,3)$. The lattice $S$ is isometric to: 
$$
U\oplus E_8(-1)\oplus H_5\oplus A_4(-1)^{\oplus 2},
$$ 
so the discriminant
form of $S$ is isomorphic to: 
$$
\frac{\IZ}{5\IZ}\left(\frac{2}{5}\right)\oplus \frac{\IZ}{5\IZ}\left(\frac{-4}{5}\right) \oplus \frac{\IZ}{5\IZ}\left(\frac{-4}{5}\right).
$$ 
To prove that this lattice admits a primitive embedding in $L\cong\HH^2(X,\IZ)$, we apply Nikulin's results on primitive embeddings~\cite[Proposition~1.15.1]{Nikulin}, here in a quite easy situation. By~\cite[Proposition~2.7 and its proof]{BCS_class}, the existence of a primitive embedding of~$S$ in~$L$ is equivalent to the existence of a lattice whose signature and discriminant form are those of an orthogonal complement of~$S$ in~$L$, here a lattice of signature $(1,2)$ and discriminant form 
isomorphic to:
$$
\frac{\IZ}{5\IZ}\left(\frac{-2}{5}\right)\oplus \frac{\IZ}{5\IZ}\left(\frac{4}{5}\right) \oplus \frac{\IZ}{5\IZ}\left(\frac{4}{5}\right)\oplus \frac{\IZ}{2\IZ}\left(\frac{-1}{2}\right).
$$
An easy computation shows that the lattice $U(5)\oplus\langle -10\rangle$ satisfies these requirements, so $S$ admits a primitive embedding in $L$ and, again by~\cite[Proposition~2.7]{BCS_class}, the embedding is unique up to an isometry of~$L$.
\end{proof}

Using~\cite[Theorem~5.3]{AST} we see that each case of the classification but the case $(m,a)=(5,3)$ is geometrically realized by an automorphism of the Hilbert square of a K3 surface,
induced by an order five nonsymplectic automorphism of the underlying surface. We refer to~\cite[\S 6.1]{BCS_class} for a geometric description of these automorphisms.
Following the terminology of~\cite[Definition~4.1]{BCS_class}, for $X$~deformation equivalent to the Hilbert square~$\Sigma^{[2]}$ of a K3 surface~$\Sigma$, an automorphism~$f$ of~$X$ is called \emph{natural} if the pair $(X,f)$ deforms to a pair
$(\Sigma^{[2]},\sigma^{[2]})$, where $\sigma$~is an automorphism of~$\Sigma$.
One of the main questions in the study of automorphisms of this kind of varieties is to find out whether they admit nonnatural automorphisms, meaning that they have nonexpected symmetries beyond those of the geometric objects they are built on. Natural automorphisms are easy to detect from the isometry class of their invariant lattice, whose orthogonal decomposition has a $\langle -2 \rangle$ factor. In a similar spirit, following 
the terminology of~\cite{CKKM,MW}, an automorphism is called \emph{(twisted) induced} if it
deforms to an automorphism of a (twisted) moduli space of stable sheaves over a K3 surface, with fixed Mukai vector, where the automorphism comes from the underlying K3 surface. This construction permits to realize geometrically more examples whose existence is proven using lattice theory.

\begin{example}  The following geometric realization of the case $(m,a)=(5,3)$ has been communicated to us by Chiara Camere. We refer to~\cite[\S 2.3]{CKKM} for an overview on moduli spaces of twisted sheaves; we follow the presentation given there. Similar examples are given in~\cite{CC}.
 By~\cite[\S 5]{AST}, there exists a projective K3 surface~$S$ with an order five nonsymplectic automorphism~$\sigma$ with invariant lattice isometric to the lattice~$H_5$. Furthermore, denoting by~$H$ the generator with square $H^2=2$ of this lattice, we may assume that $H$~is ample. Generically in the moduli space of such pairs~$(S,\sigma)$, the Picard group of~$S$ is equal to the invariant lattice of~$\sigma$, hence isometric to the lattice~$H_5$, and its transcendental lattice~$\Trans(S)$ is thus isometric to:
\[
H_5\oplus U\oplus E_8(-1)^{\oplus 2}.
\] 
Denote by $e$~the first generator of the summand~$U$ in this decomposition. The surjective morphism $\alpha\colon \Trans(S)\to \IZ/5\IZ$ given by intersecting with~$e$ defines an order~$5$ element in the Brauer group $\HH^2(\cO_S^\ast)$ of~$S$, that we still denote~$\alpha$. Consider the Mukai lattice:
\[
\HH^\ast(S,\IZ)\coloneqq\HH^0(S,\IZ)\oplus \HH^2(S,\IZ)\oplus \HH^4(S,\IZ),
\]
equipped with the pairing $(r,H,s)\cdot (r',H',s')=H\cdot H'-rs'-r's$. We take the $B$-field $B\coloneqq \frac{e}{5}$ and the twisted Mukai vector $v_B\coloneqq (0,H,0)\in \HH^\ast(S,\IZ)$. The moduli space $M_{v_B}(S,\alpha)$ of stable $\alpha$-twisted coherent sheaves on~$S$, with twisted Mukai vector~$v_B$ is an IHS manifold deformation equivalent to the Hilbert square of a K3 surface, and $\sigma$~induces an order five nonsymplectic automorphism~$f$ on it. We want to check that this automorphism corresponds to the case $(m,a)=(5,3)$. For this we compute its invariant lattice, which is generically the intersection of~$v_B^\perp$ with the sublattice of the Mukai lattice generated by $\Pic(S)$ and the classes~$(0,0,1)$ and~$(5,e,0)$. An easy computation shows that these last two classes are orthogonal to $\Pic(S)$ and generate a sublattice isometric to~$U(5)$ and contained in~$v_B^\perp$, so the computation reduces to computing the orthogonal complement of the class~$H$  in~$\Pic(S)$, which is a lattice isometric to~$\langle -10\rangle$. As a consequence, the invariant lattice of the action of~$\sigma$ on~$M_{v_B}(S,\alpha)$  is isometric to $U(5)\oplus\langle -10\rangle$ and thus corresponds to the case $(m,a)=(5,3)$.
\end{example}

As a conclusion, the following result shows that not only each case of the classification can be geometrically realize as above, but any variety with such an automorphism deforms equivariantly to those. This answers the open question~\cite[Remark~6.1]{BCS_class}:

\begin{corollary} 
Every order five nonsymplectic automorphism, acting on an irreducible holomorphic symplectic manifold deformation equivalent to the Hilbert square of a K3 surface is natural if $(m,a)\neq (5,3)$ and is twisted induced otherwise.
\end{corollary}

\begin{proof}
 Assume first that $(m,a)\neq (5,3)$. By~\cite[Corollary~5.7]{BCS_class}, the action $\varphi\coloneqq f^\ast$ of the automorphism on $\HH^2(X,\IZ)$, 
hence $f$~itself, is uniquely determined by the lattices~$S$ and~$T$.
By~\cite[Theorem~4.5, Section~5.3, Theorem~5.6, Section~7.1]{BCS_ball} the moduli spaces of lattice polarized nonsymplectic automorphisms considered here are connected
(in the case $(m,a)=(1,1)$ it is a point), and since each case is already realized by an automorphism coming from a K3 surface, this proves the statement. 

A similar argument  shows that, in the case $(m,a)=(5,3)$, every automorphism deforms to a twisted induced one. The tricky point consists in proving that the lattice~$S$ given in the table determines the action of~$\varphi$, up to conjugation by an isometry of~$S$. For this we use a recent work of Brandhorst--Cattaneo~\cite{BC}. The order five isometry $\varphi=f^\ast$ acts fixed point free on the lattice~$S$, so its minimal polynomial is~$\Phi_5$. The minimal polynomial of $\varphi+\varphi^{-1}$ is thus $X^2+X-1$, so its action on $S_\IR\coloneqq S\otimes_\IZ\IR$ admits an orthogonal decomposition:
\[
S_\IR=S_\IR(\alpha_1)\oplus S_\IR(\alpha_2),
\]
where $S_\IR(\alpha_i)$ is the eigenspace of $\varphi+\varphi^{-1}$ for the real eigenvalue $\xi_5^i+\xi_5^{-i}$. Both spaces have dimension~$10$. Since $S_\IR$ has signature $(2,18)$, the spaces $S_\IR(\alpha_i)$  have necessarily signatures $(2,8)$ and $(0,10)$. The negative parts $(8,10)$ define the \emph{signatures} of the $\IZ[\xi_5]$-lattice associated to the action of~$\varphi$ on~$S$. Taking~$\varphi^2$ instead of~$\varphi$ simply permutes the signatures. As explained in~\cite[Proposition~2.10]{BC} (here the Steinitz class is trivial), it follows that there exists a unique conjugacy class of subgroups generated by a fixed point free isometry of order five on the lattice~$S$.
\end{proof}

\section{Natural automorphisms of Kummer fourfolds}

\subsection{Generalized Kummer varieties}

Let $A$ be a two-dimensional complex torus with origin $0 \in A$. Let $\symm n A$ be the $n$\nobreakdash-th symmetric power of $A$,
$s\colon \symm n A \to A$ the summation morphism and $\pi\colon A^n \to \symm n A$ the quotient morphism.
Let $\hilb n A$ be the Hilbert scheme of $n$ points on $A$, $\rho\colon \hilb n A \to \symm n A$ the 
Hilbert--Chow morphism and put $\sigma \coloneqq s \circ \rho$.

The \emph{$n$\nobreakdash-th generalized Kummer variety of $A$} is the fibre $\genKm \coloneqq \sigma^{-1}(0)$.
It fits into the following cartesian diagram:
\begin{equation}\label{diagr:Kummer}
\xymatrix{A\times\genKm\ar[r]^-\nu\ar[d]_p & \hilbA\ar[d]^\sigma\\A\ar[r]_{\nn} & A}
\end{equation}
where $p$ is the projection onto the first factor, $\nn\colon A \to A$ is multiplication by~$n$ and $\nu\colon A \times \genKm \to \hilb n A,
(a, \xi) \mapsto a + \xi$ is a Galois covering with Galois group $G_n = \Tors_n(A)$, the finite abelian group of $n$-torsion points on
$A$. Here, $G_n$ acts  on $A \times \genKm$ by $g \cdot (a, \xi) = (a - g, g + \xi)$.
The variety $\genKm$ is a projective irreducible holomorphic symplectic manifold of dimension $2n-2$.

\subsection{Topological Lefschetz numbers}

Let $\psi\colon A \to A$ be an automorphism of the complex manifold $A$. It decomposes as
$\psi(x)= b + h(x)=(t_b\circ h)(x)$ for some $b \in A$, where $t_b$ denotes the translation by $b$ and $h\colon A \to A$
 is an isomorphism of complex Lie groups, both uniquely determined. It naturally induces
an isomorphism $\hilb n \psi\colon \hilb n A \to \hilb n A$. If we further assume that $b \in\Tors_n(A)$, the restriction of
$\hilb n \psi$ induces an automorphism $\ps n \psi \colon \genKm \to \genKm$.

We are interested in the topological Lefschetz number:
$$
	L(\ps n \psi) = \sum_{k\geq 0} (-1)^k \trace \left((\ps n \psi)^*|_{\HH^k(\genKm, \IC)}\right).
$$
The rough idea (see~\cite{NW2} for a similar idea) is to compute it by the relation: 
$$
L(\psi)L(\ps n \psi)=L(\psi \times \ps n \psi).
$$ 
Since $L(\psi)$
 may be zero, we replace the Lefschetz number
$L(f)$ of an automorphism $f\colon X \to X$ of a projective variety $X$ by the polynomial trace:
$$
	L(f, q) \coloneqq \sum_{k\geq 0}  (-1)^k \trace \left(f^*|_{\HH^k(X, \IC)}\right) q^k,
$$
which is always invertible in $\IC\llbracket q\rrbracket$. Evaluating at $q = 1$, we recover the usual topological Lefschetz number.

In our situation, since translations on $A$ are homotopic to the identity, they act trivially on the cohomology groups of $A$, hence $L(\psi,q) = L(h,q)$.
But for $n\geq 2$, in general $L(\psi^{\llbracket n\rrbracket})\neq L(h^{\llbracket n\rrbracket})$, since although translations by $n$-torsion points do act trivially on $\HH^2(\genKm,\IC)$ (see \cite[Corollary~5]{BNWS_Enriques})
they act nontrivially on the whole cohomology space of $\genKm$, as observed in~\cite[Theorem~1.3]{Oguiso}.

The third author noticed that there is a mistake in the proof of~\cite[Proposition~7]{BNWS_Enriques} giving a formula for the Lefschetz number of natural automorphisms of generalized Kummer varieties. Actually, the formula itself is not always true when the automorphism contains a nontrivial translation. In the sequel we state the correct formula and we give a complete proof.

\begin{proposition}\label{prop:formule}
Let $\psi\in\Aut(A)$, decomposed as $\psi=t_b\circ h$ with $b\in \Tors_n(A)$, and $\Psi\coloneqq \psi^*\colon \HH^1(A,\IC)\to \HH^1(A,\IC)$. Then:
$$
\sum_{n\geq 0} L(\ps n \psi,q)q^{-2n}t^n=\frac{1}{L(\psi,q)}\sum_{\chi \in (\dual G_n)^h} \chi(b)\prod_{v \geq 1}\prod_{i=0}^4\left(\det(1-(\wedge^i\Psi)q^{i-2}t^{v\abs\chi})\right)^{(-1)^{i+1}}.
$$
In particular,
$$
L(\ps n \psi,q)=\frac{q^{2n}}{L(\psi,q)}  \left.\frac{\diff^n}{n! \diff\! t^n}\right|_{t = 0}
		\sum_{\chi \in (\dual G_n)^h} \chi(b)
		\prod_{v \ge 1}
		\prod_{i=0}^4\left(\det(1-(\wedge^i\Psi)q^{i-2}t^{v\abs\chi})\right)^{(-1)^{i+1}}.
$$
\end{proposition}

\begin{proof}
The morphism $\nu$ is not natural with respect to the actions induced by~$\psi$
since
$\nu\circ(\psi\times\ps n \psi)\neq \hilb n \psi\circ\nu$, but  it commutes with
the action of ${h\times\ps n \psi}$ on ${A\times\genKm}$ and of $\hilb n \psi$ on $\hilb n A$;
this is enough for our purpose since ${L(\psi,q)=L(h,q)}$.
Since $\nu$ is a finite morphism, its higher direct images vanish, so the Leray--Serre spectral sequence yields
a canonical isomorphism: 
$$
\HH^\ast(A\times\genKm,\IC)\cong \HH^\ast(\hilb n A,\nu_\ast\underline{\IC}),
$$ 
which is compatible
with the actions of $h\times\ps n \psi$ and  $\hilb n \psi$, where the action of $\hilb n \psi$ on $\nu_\ast\underline{\IC}$
is given at each fiber $\zeta\in\hilb n A$ by: 
$$
(\hilb n \psi)^\ast_\zeta\colon (\nu_\ast\underline{\IC})(\zeta)\to (\nu_\ast\underline{\IC})(\hilb n \psi(\zeta)), \quad f\mapsto f\circ (h\times \ps n \psi).
$$

Since $\nu$ is a Galois covering, it is a classical result that $\nu_\ast\underline{\IC}$ decomposes as a direct sum of character sheaves over the character group $\dual G_n$ of 
$G_n=\Tors_n(A)$:
$$
\nu_\ast\underline{\IC}=\bigoplus_{\chi\in \dual G_n}L_{\hilb n A,\chi}
$$
where $L_{\hilb n A,\chi}$ is the locally constant sheaf on $\hilb n A$ with fibre $\IC$, associated to the character~$\chi$, whose sections over an open set $U\subset\hilb n A$
are the continuous  functions $f\colon \nu^{-1}(U)\to \IC$ (where $\IC$ is given the discrete topology)
such that $f(g\cdot (a,\xi))=\chi(g)f(a,\xi)$ for all $(a,\xi)\in\nu^{-1}(U)$  and $g\in G_n$. It follows 
(see also~\cite[Proposition~18]{Britze}) that:
$$
\HH^\ast(A\times\genKm,\IC)\cong\bigoplus_{\chi\in \dual G_n}\HH^\ast(\hilb n A,L_{\hilb n A,\chi}).
$$

The direct image $\nn_\ast\underline{\IC}$ decomposes 
similarly
as a direct sum of character sheaves over the character group $\dual G_n$ of 
$G_n=\Tors_n(A)$:
$$
\nn_\ast\underline{\IC}=\bigoplus_{\chi\in \dual G_n}L_{A,\chi}
$$
where $L_{A,\chi}$ is the locally constant sheaf on $A$ with fibre $\IC$, associated to the character~$\chi$. Here we make the group $G_n$ act by substraction on $A$, so the morphism $p$ is $G_n$-equivariant.
With this convention, the sections of $L_{A,\chi}$ over an open set $U\subset A$
are the continuous functions $f\colon \nn^{-1}(U)\to \IC$ 
such that $f(g\cdot x)=f(x-g)=\chi(g)f(x)$ for all $x\in\nn^{-1}(U)$ and $g\in G_n$.

Following~\cite{NW}, we denote by $\cL_\chi$ the locally constant sheaf on $A^{(n)}$ such that $\pi^{-1}\cL_\chi\cong L_{A,\chi}\boxtimes\cdots\boxtimes L_{A,\chi}$,
and we put $\hilb n L_\chi\coloneqq \rho^{-1}\cL_\chi$. We construct, for any $\chi\in \dual G_n$, an isomorphism of local systems $\hilb n L_\chi\xrightarrow{\sim} L_{\hilb n A,\chi}$ as follows.
Given a point $\zeta=(\zeta_1,\ldots,\zeta_n)\in\hilb n A$ and functions $f_i\in L_{A,\chi}(\zeta_i)$, $i=1,\ldots,n$ defining a section $f_1\otimes\cdots\otimes f_n\in \hilb n L_\chi(\zeta)$,
we associate a function $f\colon\nu^{-1}(\zeta)\to\IC$ defining a section $f\in  L_{\hilb n A,\chi}(\zeta)$. At a point $(a,\xi)\in\nu^{-1}(\zeta)$, with $\xi=(x_1,\ldots,x_n)$, we have $\zeta=a+\xi$,
$\zeta_i=a+x_i$ for all $i$ and $\sum\limits_{i=1}^n x_i=0$. Choose elements $\bar a, \bar x_1,\ldots,\bar x_{n-1}\in A$ such that $n\bar a =a$, $n\bar x_i=x_i$ for $i=1,\ldots,n-1$ and put
$\bar x_n\coloneqq -\sum\limits_{i=1}^{n-1} \bar x_i\in\nn^{-1}(x_n)$. We define:
$$
f(a,\xi)\coloneqq \prod_{i=1}^n f_i(\bar a+\bar x_i).
$$
It is easy to check that this definition does not depend on the choice of $\bar a,\bar x_1,\ldots,\bar x_{n-1}$, that
$f\in L_{\hilb n A,\chi}(\zeta)$, {\it i.e.}:
$$
f\left(g\cdot(a,\xi)\right)=\chi(g)f(a,\xi)=f(a-g,g+\xi)\quad \forall g\in G_n,
$$
and that this construction defines an isomorphism $\hilb n L_\chi(\zeta)\xrightarrow{\sim} L_{\hilb n A,\chi}(\zeta)$.

Recall that $\hilb n \psi$ acts on $\nu_\ast\underline{\IC}$. It is easy to check that if $f\in L_{\hilb n A,\chi}$ then  $(\hilb n \psi)^\ast(f)=f\circ(h\times\ps n \psi)\in L_{\hilb n A,\chi\circ h}$.
 Let us compute the action induced on $\bigoplus\limits_{\chi\in \dual G_n}\hilb n L_\chi$ by the isomorphisms $\hilb n L_\chi\xrightarrow{\sim} L_{\hilb n A,\chi}$.
At a point $(a,\xi)\in A\times \ps n A$, using the same notation as above, we compute:
\begin{align*}
((\hilb n \psi)^\ast f)(a,\xi)&=f\left(h(a),b+h(\xi)\right)\\
&=f\left(h(a),(b+h(x_1),\ldots,b+h(x_n))\right)\\
&=\prod_{i=1}^{n-1} f_i\left(h(\bar a)+\bar b+h(\bar x_i)\right)\cdot f_n\left(h(\bar a)+\bar b+h(\bar x_n)-b\right)
\end{align*}
where the element $\bar b+h(\bar x_n)-b\in\nn^{-1}(b+h(x_n))$ is chosen according to the construction explained above.
Thus:
\begin{equation}\label{the_key}
((\hilb n \psi)^\ast f)(a,\xi)=\chi(b)\prod_{i=1}^n (f_i\circ\psi)(\bar a +\bar x_i).
\end{equation}
This computation implies the following. Consider the action of the automorphism~$\psi$ on $L_{A,\chi}$ defined at any point $x\in A$ and for any
element $f\in L_{A,\chi}(x)$, by the formula $\psi^\ast f=f\circ\psi$. It is easy to check that,
 if $f\in L_{A,\chi}(x)$, then $\psi^\ast f\in L_{A,\chi\circ h}(h^{-1}(x))$. Let us denote by $\hilb n {(\psi^\ast)}$ the action of $\psi$ induced on $\hilb n L_\chi$, such that:
$$
\hilb n {(\psi^\ast)}(f_1\otimes\cdots\otimes f_n)=\psi^\ast f_1\otimes\cdots\otimes \psi^\ast f_n.
$$
Equation \eqref{the_key} then shows:
$$
(\hilb n \psi)^\ast (f)=\chi(b)\hilb n {(\psi^\ast)} (f) \in L_{\hilb n A,\chi\circ h} \quad \forall f\in L_{\hilb n A,\chi}.
$$
In other words, the action $(\hilb n \psi)^\ast\colon L_{\hilb n A,\chi}\to L_{\hilb n A,\chi\circ h}$ corresponds through the isomorphisms $\hilb n L_\chi\xrightarrow{\sim} L_{\hilb n A,\chi}$ 
to the action $\chi(b)\hilb n {(\psi^\ast)} \colon \hilb n L_\chi\to \hilb n L_{\chi\circ h}$.

Using~\cite[Theorem~1.2]{NW}, we get linear isomorphisms:
\begin{align*}
\HH^\ast(A\times\genKm,\IC[2n])&\cong\bigoplus_{\chi\in \dual G_n}\HH^\ast(\hilb n A,L_{\hilb n A,\chi}[2n])\\
&\cong\bigoplus_{\chi\in \dual G_n}\HH^\ast(\hilb n A,\hilb n L_\chi[2n])\\
&\cong\bigoplus_{\chi\in \dual G_n} \Sym^\ast\left(\bigoplus_{\ell\geq 1}\HH^\ast(A,L_\chi^{\otimes \ell}[2])\right),
\end{align*}
such that the action of $h\times \ps n \psi$ on the left hand side corresponds to the action of $\chi(b)\hilb n {(\psi^\ast)}$ on the second line, and acts on the last line by $\psi^\ast$ on each cohomological group, followed by a multiplication
by $\chi(b)$ on each super-symmetric tensor.

An element of $\HH^k(A\times\genKm,\IC[2n])$ has \emph{cohomological degree} $k-2n$ and \emph{conformal weight} $n$,
an element of $\HH^k(A,L_\chi^{\otimes \ell}[2])$ has \emph{degree} $k-2$ and \emph{weight} $\ell$. Notation~$\Sym^\ast$ denotes the super-symmetric algebra, where
the super-structure concerns only the cohomological degree: the weighting does not interfere with the super-structure. The above isomorphisms respect these bigradings. 

As explained in~\cite[p.~768]{NW}, one has
$\HH^\ast(A,L_\chi^{\otimes \ell}[2])=0$ unless $L_\chi^{\otimes \ell}=\underline{\IC}$.
In order to keep track of the weighting, we work in the space of formal series in the parameter~$t$ encoding the weight, with coefficients
in the space of finitely dimensional graded super-vector spaces. We thus have:
$$
\bigoplus_{\ell\geq 1}\HH^\ast(A,L_\chi^{\otimes \ell}[2]) t^\ell\cong \bigoplus_{\ell\geq 1}\HH^\ast(A,\IC[2])t^{\ell\abs\chi},
$$
where $\abs \chi$ denotes the order of $\chi$ in $\dual G_n$.
We get finally:
$$
\bigoplus_{n\geq 0} \HH^\ast(A\times\genKm,\IC[2n])t^n\cong \bigoplus_{\chi\in \dual G_n} \Sym^\ast\left(\bigoplus_{\ell\geq 1}\HH^\ast(A,\IC[2]) t^{\ell|\chi|}\right).
$$

Let us now compute the formal series of polynomial traces $\sum\limits_{n\geq 0} L(h\times\ps n \psi,q)t^n$ using this isomorphism. First we observe that the action on the right hand side is induced by the natural
action of $\psi^\ast$, but it sends the block indexed by $\chi$ to the block indexed by $\chi\circ h$. So in the computation of the trace, only those characters which are invariant by $h$ contribute to the trace,
we denote them $\chi\in(\dual G_n)^h$. Once the trace of $\psi^\ast$ is computed, it is multiplied by $\chi(b)$ on each block $\chi$ to take into account our previous computation of the action.
Now the trace of $\psi^\ast$ on each super--symmetric block can be computed using standard techniques for traces of natural operators on bigraded super-vector spaces, for which we refer 
to~\cite[Section~3]{Boissiere} and~\cite[Lemma~6]{BNWS_Enriques}. Denoting by $\Psi$ the action of $\psi^\ast$ on $\HH^1(A,\IC)$, and using the fact that $\HH^\ast(A,\IC)=\bigwedge^\ast\HH^1(A,\IC)$,
we obtain:
$$
\sum_{n\geq 0} L(h\times\ps n \psi,q)q^{-2n}t^n=\sum_{\chi \in (\dual G_n)^h} \chi(b)\prod_{v \geq 1}\prod_{i=0}^4\left(\det(1-(\wedge^i\Psi)q^{i-2}t^{v\abs\chi})\right)^{(-1)^{i+1}}.
$$
Since $L(h\times\ps n \psi,q)=L(h,q)\cdot L(\ps n \psi,q)=L(\psi,q)\cdot L(\ps n \psi)$ we get the expected formula.
\end{proof}

\begin{remark}
The isomorphisms of local systems $\hilb n L_\chi\xrightarrow{\sim} L_{\hilb n A,\chi}$ constructed in the proof for any $\chi\in \dual G_n$ can be understood more conceptually as follows.
Since diagram~(\ref{diagr:Kummer}) is $G_n$-equivariant, cartesian, and since $\nn$ is finite and unramified, we have an isomorphism of $G_n$-local systems
(see also~\cite[Proposition~18]{Britze}):
$$
\nu_\ast\underline{\IC}\cong \nu_\ast p^{-1}\underline{\IC}\cong \sigma^{-1}\nn_\ast\underline{\IC}.
$$
Following these isomorphisms, it is easy to see that 
 $L_{\hilb n A,\chi}\cong\sigma^{-1}L_{A,\chi}$ for all $\chi\in \dual G_n$.
 We have $\sigma^{-1}L_{A,\chi}=\rho^{-1}s^{-1}L_{A,\chi}$,
and it is easy to see that:
$$
s^{-1}L_{A,\chi}\cong \cL_\chi,
$$
so finally $L_{\hilb n A,\chi}\cong \sigma^{-1}L_{A,\chi}\cong \rho^{-1}\cL_\chi=\hilb n L_\chi$.
\end{remark}

As a consequence of Proposition~\ref{prop:formule}, evaluating at $q=1$ we recover the formula stated in~\cite[Proposition~7]{BNWS_Enriques}, this time with the missing factors:

\begin{corollary}\label{cor:formula}
Let $\psi\in\Aut(A)$, decomposed as $\psi=t_b\circ h$ with $b\in \Tors_n(A)$, and $\Psi\coloneqq \psi^*\colon \HH^1(A,\IC)\to \HH^1(A,\IC)$. Then:
$$
	\begin{aligned}
		L(\psi)L(\ps n \psi)
		& =  \left.\frac{\diff^n}{n! \diff\! t^n}\right|_{t = 0}
		\sum_{\chi \in (\dual G_n)^h} \chi(b)
		\prod_{v \ge 1}
		\exp\left(\sum_{s \ge 1} \frac{\det(1 - \Psi^s)}{s} t^{v \abs\chi s}\right).
	\end{aligned}
$$
\end{corollary}

\subsection{Application}
\label{ss:appl}

Automorphisms of a two dimensional complex torus~$A$ were classified by Fujiki~\cite{Fujiki}. From this we extract a list of those automorphisms whose action on 
the second cohomology space $\HH^2(A,\IC)$ has prime order, and we consider the automorphisms 
induced on the generalized Kummer fourfold $\ps 3 A$. To get a complete picture, we consider
all automorphisms~$\ps 3\psi$ with $\psi=t_b\circ (\pm h)$, where~$h$ is a group automorphism of prime order on~$A$ and $t_b$~is the translation by a $3$-torsion point $b\in A$. 
As an application of our formula in Proposition~\ref{prop:formule}, we give in each case the topological Lefschetz number and we briefly discuss the fix locus.
Details on the geometric study can be found in~\cite[Section~1.2]{Tari}.
We refer to~\cite[Lemma~1.13]{Tari} for a discussion of the fix loci of the iterates of the automorphisms listed below.
In complement, and in relation to the first part of this note, we refer to~\cite{MTW_Aut} for a lattice-theoretical classification of these automorphisms.

\subsubsection{Type 0} Assume that $h=\id$. Clearly $L(\ps 3 \id)=108$. For $b\neq 0$ we get $L(\ps 3 t_b)=27$, the fix locus consists of $27$~points.

 We have $L(\ps 3 {(-\id)})=60$, the fix locus consists of a copy of~$\ps 2 A$ and $36$ isolated points; the same result holds for $\ps 3 {(-t_b)}$ if $b\neq 0$. 

\subsubsection{Type 1} The torus $A=E\times E'$ is a product of two elliptic curves and ${h=\id_E\times (-\id_{E'})}$. We have $L(\ps 3 h)=12$, the fix locus consists of a copy of $E\times\IP^1$ blown up in $9$ points, three copies of $E^{(2)}$, one copy of $\ps 3 E$ and one copy of $E\times E$. 

Let $b=(u,v)$. If $u=0$ then $L(\ps 3 \psi)=12$ and the fix locus is the same as above. If $u\neq 0$, we have $L(\ps 3\psi)=3$, the fix locus consists of $3$ points.

\subsubsection{Type 2} The torus $A=\frac{E\times E'}{\IZ/2\IZ}$ is the quotient of $E\times E'$ by the translation by a $2$-torsion point $(a,a')$ with $a\in\Tors_2(E)$, $a'\in\Tors_2(E')$, both nonzero, and ${h=\id_E\times (-\id_{E'})}$. We have $L(\ps 3 h)=12$, the fix locus consists of a copy of $A/\langle h\rangle$ blown up in $9$ points, one copy of $E^{(2)}$ and one copy of $\ps 3 E$. 

 Let $b=(u,v)$. If $u=0$ then $L(\ps 3 \psi)=12$ and the fix locus is the same as above. If $u\neq 0$, we have $L(\ps 3\psi)=3$ and the fix locus consists of $3$ points.

\subsubsection{Type 3} The torus is $A=\frac{E\times E'}{(\IZ/2\IZ)^2}$ where $(\IZ/2\IZ)^2\to E\times E'$ is any group monomorphism, and $h=\id_E\times (-\id_{E'})$. We have $L(\ps 3 h)=12$, the fix locus consists of a copy of $A/\langle h\rangle$ blown up in $9$ points and one copy of $\ps 3 E$. 

Let $b=(u,v)$. If $u=0$ then $L(\ps 3 \psi)=12$ and the fix locus is the same as above. If $u\neq 0$, we have $L(\ps 3\psi)=3$ and the fix locus consists of $3$ points.

\subsubsection{Type 4} The torus is $A=E_4\times E_4$ where $E_4=\frac{\IC}{\IZ\oplus\ii\IZ}$, and $h=\begin{pmatrix} \ii &0\\ 0&\ii\end{pmatrix}$. We have  $L(\ps 3 h)=16$, the fix locus consists of four copies of $\IP^1$ and $8$ isolated points.  The same result holds when $h$ is composed by a nonzero  translation.

\subsubsection{Type 5} The torus is $A=E\times E_6$ where $E_6=\frac{\IC}{\IZ\oplus\zeta_6\IZ}$, with $\zeta_n$ is a primitive $n$-th root of  unity, and $h=\begin{pmatrix} 1 &0\\ 0&\zeta_3\end{pmatrix}$. We have  $L(\ps 3 h)=27$, the fix locus consists of nine copies of $\IP^1$, one copy of $E^2$, three copies of $\ps 3 E$ and three copies of $E$. 

Let $b=(u,v)\in\Tors_3(A)$. Denote $\Delta_6$ the image in $E_6$ of the real line $\IR\frac{1+\zeta_6}{3}$. If  $u=0$ and $v\in\Delta_6$, then $L(\ps 3 \psi)=27$, the fix locus consists of nine copies of $\IP^1$, one copy of $E^2$, three copies of $\ps 3 E$ and three copies of $E$. Otherwise $L(\ps 3 \psi)=0$ but the fix locus depends on the translation. If $u=0$ and $v\notin\Delta_6$, it consists of three copies of $E^{(2)}$ and three copies of $E$; if $u\neq 0$ and $v\in\Delta_6$, it consists of three copies of $E_6$; if $u\neq 0$ and $v\notin\Delta_6$ the fix locus is empty: this case already appeared in~\cite[Proposition~8]{BNWS_Enriques} in the context of the construction of Enriques fourfolds. 

We have $L(\ps 3 {(-h)})=9$, the fix locus consists of  one copy of $E$, one copy of $\IP^1$ and $7$ isolated points. The same result holds when $(-h)$ is composed by a nonzero  translation.

\subsubsection{Type 6}  The torus $A=\frac{E\times E_6}{\IZ/3\IZ}$ is the quotient of $E\times E_6$ by the translation by a $3$-torsion point $(a,a')$ with $a\in\Tors_3(E)\setminus\{0\}$ and $a'=\frac{1+\zeta_6}{3}$, and $h=\begin{pmatrix} 1&0\\ 0 & \zeta_3\end{pmatrix}$. We have $L(\ps 3 h)=9$, the fix locus consists of three copies of $\IP^1$, one copy of $\ps 3 E$ and one copy of $E$.

Let $b\in\Tors_3(A)$, it can be written as $b=(u,v)+\frac{t}{3}(a,a')$ with $u\in\Tors_3(E)$, $v\in\Tors_3(E_6)$ and $t\in\{0,1,2\}$. If $t=0$ and $u\in\IZ a$ then 
$L(\ps 3\psi)=9$, the fix locus consists of one copy of $\ps 3 E$, three copies of $\IP^1$ and one copy of $E$. Otherwise $L(\ps 3\psi)=0$ but the fix locus depends on the translation. If $t=0$ and $u\notin\IZ a$, it consists of one copy of $E_6$, if $t\neq 0$ the fix locus is empty.

We have $L(\ps 3 {(-h)})=9$, the fix locus consists of one copy of $\IP^1$ and $7$ isolated points. The same result holds when $(-h)$ is composed by a nonzero  translation.

\subsubsection{Type 7} The torus is $A=E_6\times E_6$ and $h=\begin{pmatrix} \zeta_3 & 0\\ 0&\zeta_3\end{pmatrix}$. Here $L(\ps 3 h)=36$, the fix locus consists of a copy of the minimal resolution of $A/\langle h\rangle$ and $21$ isolated points.

Let $b=(u,v)$. If $b\in \Delta_6\times \Delta_6$, then $L(\ps 3 \psi)=36$ and the fix locus is the same as above, otherwise $L(\ps 3 \psi)=27$ and it consists of nine copies of $\IP^1$ and $9$ isolated points.

We have $L(\ps 3 {(-h)})=12$, the fix locus consists of one copy of $\IP^1$ and $10$ isolated points. The same result holds when $(-h)$ is composed by a nonzero  translation.

\subsubsection{Type 8} The torus is $A=\IC^2/\Lambda$ where $\Lambda$ is the lattice generated by the four vectors:
$$
(1,1),(\zeta_5,\zeta_5^2),(\zeta_5^2,\zeta_5^4),(\zeta_5^3,\zeta_5),
$$
and $h=\begin{pmatrix} \zeta_5 & 0 \\ 0& \zeta_5^2\end{pmatrix}$. We have $L(\ps 3 h)=13$, the fix locus consists of one copy of $\IP^1$ and $11$ isolated points. The same result holds when $h$ is composed by a nonzero  translation.

We have $L(\ps 3 {(-h)})=5$, the fix locus consists of $5$ points. The same result holds when $(-h)$ is composed by a nonzero  translation.

\bibliographystyle{amsplain}
\bibliography{Biblio}

\end{document}